\documentclass[preprint,10pt]{elsarticle}
\usepackage[utf8]{inputenc}
\usepackage[english]{babel}
\usepackage[english]{babel}
\usepackage{amsfonts}
\usepackage{xcolor}
\usepackage{marginnote}
\usepackage{amsmath}
\usepackage{a4wide}
\numberwithin{equation}{section}

\usepackage{accents}

\usepackage{amssymb}
\usepackage{amsthm}

\newtheorem{thm}{Theorem}[section]

\newtheorem{cor}{Corollary}[section]
\newtheorem{prop}{Proposition}[section]

\newtheorem{rem}{Remark}[section]

\newtheorem{example}{Example}

\journal{}

\begin{document}
	
	\begin{frontmatter}
		
		\title{A new proof of the Perron–Frobenius theorem: a variational approach}
		
	
	%
	

	\author[rv]{Yavdat~Il'yasov} 
	\ead{ilyasov02@gmail.com}
	\author[rv]{Nurmukhamet Valeev} 
	\ead{valeevnf@mail.ru}
	
	\address[rv]{Institute of Mathematics of UFRC RAS, 112, Chernyshevsky str., 450008, Ufa,
		Russia}

	
	%
	
	\begin{abstract}
		
		We develop the Perron–Frobenius theory using a variational approach and extend it to a set of arbitrary matrices, including those that are neither irreducible nor essentially positive, and non-preserved cones.
		We introduce a new concept called a "quasi-eigenvalue of a matrix", which is invariant under orthogonal transformations of variables, and  has various useful properties, such as determining the largest value of the real parts of the eigenvalues of a matrix.
		We extend Weyl's inequality for the eigenvalues to the set of arbitrary matrices and prove the new stability result to the Perron root of irreducible nonnegative matrices under arbitrary perturbations. As well as this, we obtain new types of estimates for the ranges of the sets of eigenvalues and their real parts.

	\end{abstract}


\end{frontmatter}

\section{Introduction}

The Perron–Frobenius theorem, established by  O. Perron (1907) and G. Frobenius (1912),  states that any irreducible non-negative matrix $A=(a_{i,j})$ has a  simple eigenvalue $\lambda^*(A) \in \mathbb{R}$ of the largest magnitude, moreover,  its  right and left eigenvectors $\phi^*$, $\psi^*$ are positive.
 Later, this theory 
has been extended by  M. Krein, M. Rutman (1948) to
operators on a Banach space that preserved cones. 
Another development of the theory was proposed by G. Birkhoff and R. Varga (1958), who extended the theory to the essentially positive matrices, including those for which the cone of positive vectors  $S^o_+:=\{u\in \mathbb{R}^n\mid~u_i> 0,i=1,\ldots,n\}$ is not necessarily preserved.
Furthermore, they  discovered the  following identity
\begin{equation}\label{BV}
 \lambda^*(A)=\lambda_{S_+}^*(A):=\sup _{u\in S^o_+}\inf _{v\in S^o_+}{\frac{\left\langle A u, v\right\rangle}{\left\langle u, v\right\rangle}}=\inf _{v\in S^o_+}\sup _{u\in S^o_+}{\frac{\left\langle A u, v\right\rangle}{\left\langle u, v\right\rangle}}
\end{equation} 
which is now known  as the  \textit{Birkhoff-Varga formula} \cite{BirVar}. Here $S_+=\overline{S^o_+}\setminus 0$, $\left\langle \cdot,\cdot \right\rangle$ denotes the inner product in $\mathbb{R}^n$.  The matrix $A$ is said to be irreducible if $A$ cannot be conjugated into block upper triangular form by a permutation matrix $P$, and $A$ is essentially positive matrix if it is  irreducible and has nonnegative elements off the  diagonal \cite{Varga}. Hereafter,  $\lambda_j(A) \in \mathbb{C}$,~$j=1,\ldots,n$  denotes the eigenvalues of $A$ which are repeated according to their multiplicity. For reader’s convenience, we provide the statement of   the Birkhoff–Varga theorem (see Theorem 2.1, Theorem 8.2 in \cite{Varga})



\noindent
{\bf Theorem} (\textit{Birkhoff–Varga}).
	\textit{Assume that $A$ is an essentially positive matrix. Then \eqref{BV} holds true, moreover}
		\begin{description}
			\item[a)] \textit{$\lambda_{S_+}^*(A)$ is a simple eigenvalue of $A$, and the corresponding right and left eigenvectors $\phi^*$, $\psi^*$ are positive}. 
		\item[b)] $\lambda_{S_+}^*(A)= 
		\max\{\operatorname{Re}\,\lambda_j(A),~j=1,\ldots,n\}$.
		\item[c)] \textit{If in addition, $A$ is a nonnegative matrix.  Then } $\lambda_{S_+}^*(A)=
		\max\{|\lambda_j(A)|,~j=1,\ldots,n\}$.
		\item[d)] $\lambda_{S_+}^*(A)$ increases when any element of $A$ increases. 
		\end{description}
		It is common to call an eigenvalue of a matrix the Perron root if it equals the spectral radius $\rho(A)$ of $A$. Thus, under the hypothesis of assertion c) of the theorem the quantity $\lambda_{S_+}^*(A)$  corresponds to the Perron root of $A$. At the same time, this quantity for the essentially positive matrices  is not necessarily equal to the spectral radius, moreover, it can take the values of any sign. It makes sense to call $\lambda_{S_+}^*(A)$ defined by  \eqref{BV} the \textit{quasi-Perron root}.

There are numerous applications of the Perron–Frobenius theory, see, e.g.,  \cite{Berman, Meyer2, Pillai} and references therein. 
This topic has gained our attention in light of the study initiated by the first author in  \cite{IlyasFunc,IvanIlya,IlyasChaos, il2023_bifSYST, Salazar},  where the saddle-node bifurcations of  nonlinear equations of the form $F(u,\lambda):=T(u)-\lambda G(u)=0$ are determined in terms of the   saddle points of the quotient 
$$
\mathcal{R}(u,v):=\frac{\left\langle T(u), v\right\rangle}{\left\langle G(u), v\right\rangle},~~(u,v) \in U\times V,~\left\langle G(u), v\right\rangle \neq 0.
$$
Here $U,V$ are appropriate subsets of a normed space $X$, $T,G: X \mapsto X^*$ are given maps. In particular, the saddle points of $\mathcal{R}(u,v)$  can be found  using the minimax formula (see \cite{IlyasFunc, IlyasChaos, il2023_bifSYST})
\begin{equation}\label{IG}
	\lambda^*=\sup _{u\in U}\inf _{v\in V}\mathcal{R}(u,v).
\end{equation}

The minimax formula \eqref{IG} can be justified by directly substituting the  saddle point $(u^*,v^*)$ of $\mathcal{R}(u,v)$  into \eqref{IG} if it is known in advance, see \cite{IvanIlya, IlIvan1, IlyasChaos}. It was this method that Birkhoff and Varga had used in
\cite{BirVar} to justify the minimax formula \eqref{BV} (see also  \cite{Friedland}). However, for non-linear and infinite-dimensional equations finding solutions is itself a challenging problem.
This entails the following problem
 \begin{description}
 	\item[A:] \textit{Could one find the saddle point $(u^*,v^*)$ of $\mathcal{R}(u,v)$ directly through variational problem \eqref{IG}?}. 
 \end{description}
The reader may find some answers to this problem in  \cite{bobkov, IlyasFunc, IlyasPHD, IlyasChaos, il2023_bifSYST, Salazar}, which deal with variational problems of type \eqref{IG}, including those corresponding to nonlinear partial differential equations. At the same time, this field of study is still in its early stages, and many important questions remain unanswered. In this regard,  \textit{finding the quasi-Perron root of the matrix $A$ and the  eigenvectors $\phi^*, \psi^*$  directly through variational problem \eqref{BV}} may be viewed as a first-step problem to \textbf{A}.

The extension of Kerin-Rutman's theory to cone-preserving matrices has received considerable attention (see, e.g., \cite{Barker, Birkh, Burb, Gautier, Friedland, lemmens, Tam2}). In particular, a generalization of the Birkhoff-Varga formula \eqref{BV} for cone-preserving matrices was studied in \cite{ Friedland}.
Note that \eqref{BV} implies the well-known \textit{Collatz-Wieland formula}:
\begin{equation*}
	\lambda^*_{S_+}(A)= \max_{u>0}\min_{e\in \{e_1,\ldots,e_n\}} \frac{\left\langle A u, e\right\rangle}{\left\langle u, e\right\rangle},
\end{equation*}
where  $e_1:=(1,0,\ldots,0)^T,...,e_n:=(0,\ldots,0,1)^T \in \mathbb{R}^n$ (see, e.g., \cite{IvanIlya}). Methods based on this formula, the so-called Wieland's approach, were used to generalize the Perron-Frobenius theory to a set of matrices preserving cones \cite{Barker, Marek1, Marek, Tam}.
 The first results on generalizing Perron's theorem to matrices  non-preserving cones were obtained by Birkhoff and Varga (1958). 
A remarkable feature of their work was that their proposed method, along with the Perron root made it possible to determine the largest value of the real parts of the eigenvalues of matrices. Finding the largest real part of the eigenvalues of operators is important in many problems, including the stability analysis of solutions to differential equations, detection of the Hopf bifurcation,  analysis of iterative algorithms,  etc, see, e.g., \cite{Arnold, chen,  Hale, Hassard, Krasnoselskii, Meyer, Varga}. It is worth noting that in applications, and in particular in finding bifurcations of solutions, it arises  matrices that are not essentially positive or easy to classify of their type  (see, e.g., \cite{chen, Salazar}), moreover, solutions are often sought in non-invariant cones. Nevertheless, the development of the Birkhoff-Varga approach did not receive as much attention as it should have. This naturally leads us to the following general problem: 
\begin{description}
	\item[B:] \textit{Determine the widest class of matrices and cones for which the Perron-Birkhoff theory still holds, possibly in a generalized form.}
\end{description}


Let us state  our main results. In what follows,  $(M_{n\times n}(\mathbb{R}),\|\cdot\|_M)$ denotes the space of real $n\times n$ matrices endowed by standard operator norm $\|\cdot\|_M$, $\overline{K}(\mathbb{R}^n)$ is a  set of self-dual solid convex cones in $\mathbb{R}^n$, i.e.,  $\overline{C} \in \overline{K}(\mathbb{R}^n)$ if  $\alpha u +\beta w \in \overline{C}$, $\forall u,w \in \overline{C}$, $\alpha, \beta \geq 0$,  
	  $\overline{C}^{*}=\{u\in\mathbb{R}^n\mid \langle w,u\rangle \geq 0, \forall w\in \overline{C}\}=\overline{C}$, and $C^o:=$int$\overline{C}\neq \emptyset$ (see \cite{Berman, Krein}).  We denote $C:=\overline{C}\setminus 0$,  $K(\mathbb{R}^n):=\{C\mid ~\overline{C}\in \overline{K}(\mathbb{R}^n)\}$. The subset of cones from $K(\mathbb{R}^n)$ that are the positive orthant's orthogonal transform will be denoted by $K_O(\mathbb{R}^n)$, i.e., $K_O(\mathbb{R}^n):=\{C \in K(\mathbb{R}^n)\mid C=US_+,~U \in O(n)\}$, where $O(n)$ is the group of  orthogonal matrixes. We call 
	  $$
	  \displaystyle{\lambda(u,v):=\frac{\left\langle A u, v\right\rangle}{\left\langle u, v\right\rangle}}, ~~ \langle u, v\rangle \neq 0, ~~u,v \in \mathbb{R}^n,
	  $$
	  the extended Rayleigh quotient (cf. \cite{IlyasFunc,il2023_bifSYST}). Note that  for $C \in K(\mathbb{R}^n)$, $\lambda(u,v)$ is well-defined if $ (u,v)  \in C \times C^o$ or  $ (u,v)  \in C^o \times C$. 
	  Let $X,Y$ be subsets of $\mathbb{R}^n$ such that $\lambda(u,v)$ is well-defined on  $ X \times Y$. We say that \textit{the minimax principle} for $\lambda(u,v)$ is satisfied in $X \times Y$  if 
\begin{align}\label{MminD}
-\infty<\sup_{u\in X}\inf_{v \in Y}\lambda(u,v)=\inf_{v \in Y}\sup_{u\in X}\lambda(u,v)<+\infty. 
\end{align} 
Vectors $u_C(A) \in C$ and  $v_C(A) \in C$ is said to be the \textit{right and left quasi-eigenvectors} of $A$  in $C$ if the following is fulfilled
\begin{align}\label{eq:DefQQ}
	&\overline{\lambda}_C(A):=\sup_{u\in C}\inf_{v \in C^o}\lambda(u,v)=\inf_{v \in C^o}\lambda(u_C(A),v),\\
	&\underline{\lambda}_C(A):=\inf_{v \in C}\sup_{u\in C^o}\lambda(u,v)=\sup_{u\in C^o}\lambda(u,v_C(A)). \label{eq:DefQQ2}
\end{align}
We call $\overline{\lambda}_C(A)$ the upper quasi-eigenvalue, and $\underline{\lambda}_C(A)$
the lower quasi-eigenvalue  of $A$  in $C$.  
Furthermore, if 
\begin{align*}
	&\partial_u\lambda(u_C(A),v_C(A))=0,~~~\partial_v\lambda (u_C(A),v_C(A))=0,\\ &\lambda_C(A):=\overline{\lambda}_C(A)=\underline{\lambda}_C(A)=\lambda (u_C(A),v_C(A)),
	\label{sadd}
\end{align*}
then  $(u_C(A),v_C(A))$ is said to be  \textit{saddle point}  of $\lambda(u,v)$ in $C \times C$.
Observe,  this implies that $\lambda_C(A)$ is an eigenvalue of $A$ and  $u_C(A)$, $v_C(A)$ are the associated right and left eigenvectors of $A$.

 Our first  result provides certain answers to problems \textbf{A},  \textbf{B}. 

\begin{thm}\label{theorem1} 
	\par 
	Assume that  $A \in M_{n\times n}(\mathbb{R})$, $C \in K(\mathbb{R}^n)$. 
	Then  the minimax principles for $\lambda(u,v)$ are satisfied in $C\times C^o$ and $C^o\times C$
	\begin{equation}\label{eq:mimaPrin}
		\sup_{u\in C}\inf_{v \in C^o}\lambda(u,v)=\inf_{v \in C^o}\sup_{u \in C}\lambda(u,v)~~\mbox{and}~~\sup_{u\in C^o}\inf_{v \in C}\lambda(u,v)=\inf_{v \in C}\sup_{u \in C^o}\lambda(u,v)
	\end{equation}
	 Moreover, there exist  right and left quasi-eigenvectors $u_C(A), v_C(A) \in C$ of $A$  in $C$. 
		
		If  $u_C(A) \in C^o$  $(v_C(A) \in C^o)$, then the minimax principle for $\lambda(u,v)$ is satisfied in $C^o\times C^o$, and 
	\begin{align*}
		&\overline{\lambda}_C(A)=\lambda_C(A):=\sup_{u\in C^o}\inf_{v \in C^o}\lambda(u,v)=\inf_{v \in C^o}\sup_{u \in C^o}\lambda(u,v),\\
		(&\underline{\lambda}_C(A)=\lambda_C(A):=\sup_{u\in C^o}\inf_{v \in C^o}\lambda(u,v)=\inf_{v \in C^o}\sup_{u \in C^o}\lambda(u,v)).
	\end{align*}
	Furthermore, if  $u_C(A), v_C(A) \in C^o$, then $\lambda_C(A)=\lambda(u_C(A),v_C(A))$  is eigenvalue and  $u_C(A)$, $v_C(A)$ are corresponding right and left eigenvectors of $A$ in $C$.
	 			\end{thm}

Notice that \eqref{eq:mimaPrin} implies $\overline{\lambda}_C(A)\geq \underline{\lambda}_C(A)$. Indeed, by \eqref{eq:mimaPrin}
\begin{align*}
\underline{\lambda}_C(A)=\sup_{u\in C^o}\inf_{v \in C}\lambda(u,v)\leq \sup_{u\in C^o}\inf_{v \in C^o}\lambda(u,v)\leq  \sup_{u\in C}\inf_{v \in C^o}\lambda(u,v)=\overline{\lambda}_C(A).
\end{align*}	
	
	\begin{example}
	The minimax principle for $\lambda(u,v)$ may not hold in $C^o\times C^o$. Indeed, consider 
	$$
	A=
	\left( {\begin{array}{cc}
			2 & 0\\
			0 & 1 \\
	\end{array} } \right),~ ~C=S_+:=\{x \in \mathbb{R}^2\setminus 0\mid ~x_1\geq 0,~x_2\geq 0\}.
	$$
	Observe,
	\begin{equation*}
		\sup _{u \in S_+^o}\inf _{v\in S_+^o}\frac{\left\langle A u, v\right\rangle}{\left\langle u, v\right\rangle}=	\sup _{u \in S_+^o}\inf _{v\in S_+}\frac{\left\langle A u, v\right\rangle}{\left\langle u, v\right\rangle}=1,~~\inf _{v\in S_+^o}\sup _{u \in S_+^o}\frac{\left\langle A u, v\right\rangle}{\left\langle u, v\right\rangle}=\inf _{v\in S_+^o}\sup _{u \in S_+}\frac{\left\langle A u, v\right\rangle}{\left\langle u, v\right\rangle}=2.
	\end{equation*}	
	Meanwhile, 
	$$
	\inf _{v\in S_+^o}\frac{\left\langle A u, v\right\rangle}{\left\langle u, v\right\rangle}|_{u=(1,0)^T}=2~~\Rightarrow~\sup _{u \in S_+}\inf _{v\in S_+^o}\frac{\left\langle A u, v\right\rangle}{\left\langle u, v\right\rangle}=2,
	$$
	and thus, $\overline{\lambda}_{S_+}=\sup_{u\in S_+}\inf_{v \in S_+^o}\lambda(u,v)=\inf_{v \in S_+^o}\sup_{u\in S_+}\lambda(u,v)=2$, $u_{S_+}(A)=(1,0)^T$. By the same reasoning, $\underline{\lambda}_{S_+}=\sup_{u\in S_+^o}\inf_{v \in S_+}\lambda(u,v)=\inf_{v \in S_+}\sup_{u\in S_+^o}\lambda(u,v)=1$, $v_{S_+}(A)=(0,1)^T$.
	\end{example}
	
	\begin{example}
 The  right and left quasi-eigenvectors $u_C(A), v_C(A)$ of matrices may be not eigenvectors in the usual sense, and the  quasi-eigenvalues $\overline{\lambda}_C(A)$, $\underline{\lambda}_C(A)$ may be not eigenvalues. Indeed, consider 
$$
A=
\left( {\begin{array}{cc}
		1 & -1\\
		1 & 1 \\
\end{array} } \right),~ ~C=S_+:=\{x \in \mathbb{R}^2\setminus 0\mid ~x_1\geq 0,~x_2\geq 0\}.
$$
Then $\lambda_1(A)=1-i$, $\lambda_2(A)=1+i$ with corresponding eigenvectors $\phi_1=(i,1)^T$, $(1,i)^T$. However by Theorem \ref{theorem4} below, we have  
$	\overline{\lambda}_C(A)=\underline{\lambda}_C(A)=1=\min_{i=1,2}{\operatorname{Re} \lambda_i(A)}=\max_{i=1,2}{\operatorname{Re} \lambda_i(A)}$,
with the right and left quasi-eigenvectors equal to $(1,0)^T$.
		\end{example}
		\begin{rem}
			Theorem \ref{theorem1} entails that if $C^o$  does not contain  right and left eigenvectors of $A$, then $ v_C(A) \in \partial C$ or/and $u_C(A)\in \partial C$. 
			\end{rem}	
	 	It is important to note that the concept of quasi-eigenvalues $\overline{\lambda}_C(A)$, $\underline{\lambda}_C(A)$ of the matrix $A \in M_{n\times n}(\mathbb{R})$   in $C \in K(\mathbb{R}^n)$ is an invariant under the orthogonal change of variable $x=U^Ty$, $ U \in O(n)$.  Indeed, we have  
	 \begin{align}
	 	&\overline{\lambda}_{C}(A) =\sup _{u \in C}\inf _{v\in C^o}\frac{\left\langle A u, v\right\rangle}{\left\langle u, v\right\rangle}=\sup _{u\in U^TC}\inf _{v\in U^TC^o}{\frac{\left\langle U^TA U u, v\right\rangle}{\left\langle u, v\right\rangle}}=\overline{\lambda}_{U^TC}( U^TA U),\label{inv}\\
	 	&\underline{\lambda}_{C}(A) =\sup _{u \in C^o}\inf _{v\in C}\frac{\left\langle A u, v\right\rangle}{\left\langle u, v\right\rangle}=\sup _{u\in U^TC^o}\inf _{v\in U^TC}{\frac{\left\langle U^TA U u, v\right\rangle}{\left\langle u, v\right\rangle}}=\underline{\lambda}_{U^TC}( U^TA U).\label{inv2}
	 \end{align}	
	
		Let $\lambda_i(A)$ for $i\in \{1,\ldots,n\}$ be a real eigenvalue of $A$. We denote by $\phi_{i}, \psi_{i}$  right and left eigenvectors corresponding  $\lambda_i(A)$. In what follows, we use notation $\operatorname{span}(Z)$ for the linear span of a set $Z \subset \mathbb{R}^n$.

	\begin{cor}\label{cor1}
		Let $\lambda_i(A)$ for $i\in \{1,\ldots,n\}$ be a real eigenvalue of the matrix $A$.
		\begin{description}
			\item[ {\it (i)}]  If there exists a  right (left) eigenvector of $A$ such that $\operatorname{span}(\phi_i)\cap C^o \neq \emptyset$ ($\operatorname{span}(\psi_i)\cap C^o \neq \emptyset$), then  $\lambda_i(A)\leq \underline{\lambda}_C(A)$ $(\lambda_i(A)\geq \overline{\lambda}_C(A))$. 
			\item[ {\it (ii)}]   If there exist  right and left eigenvector such that $ \operatorname{span}(\phi_i)\cap C^o\neq \emptyset$,  $\operatorname{span}(\psi_i)\cap C^o \neq \emptyset$, then   $\overline{\lambda}_C(A)=\underline{\lambda}_C(A)=\lambda_i(A)$. The inverse statement is also true.
		\end{description}
	\end{cor}

	Some statements of the  Perron–Frobenius  and Birkhoff–Varga theorems still hold  under hypothesis of Theorem \ref{theorem1} and certain additional conditions.
	
	\begin{cor}\label{cor2}
		\begin{description}
			\item[{\it (i)}] If $A\in M_{n\times n}(\mathbb{R})$ is a nonnegative matrix, then 
			\begin{equation}\label{eq:Noneg1}
				\overline{\lambda}_{S_+}(A)\geq \max\{|\lambda_j(A)|,~j=1,\ldots,n\}>0,
			\end{equation}
			moreover, if in addition $\lambda_{S_+}(A)$ is an eigenvalue of $A$, then   
			\begin{equation}\label{eq:Noneg}
				\overline{\lambda}_{S_+}(A)= \max\{|\lambda_j(A)|,~j=1,\ldots,n\}>0.
			\end{equation}

			\item[{\it (ii)}] If $A\in M_{n\times n}(\mathbb{R})$ is a  matrix with nonnegative elements off the diagonal, then	
			\begin{equation}\label{eq:Re1}
				\overline{\lambda}_{S_+}(A)\geq
				\max\{\operatorname{Re}\lambda_j(A),~j=1,\ldots,n\},
			\end{equation}
			moreover, if in addition $\lambda_{S_+}(A)$ is a real part of an eigenvalue of $A$, then 
			\begin{equation}\label{eq:Re}
				\overline{\lambda}_{S_+}(A)=
				\max\{\operatorname{Re}\lambda_j(A),~j=1,\ldots,n\}.	
			\end{equation}

		\end{description}
	\end{cor}

		\begin{cor} \label{cortheorem4} 		Assume that $A$ is a real symmetric matrix and $C\in K_O(\mathbb{R}^n)$.  Then the following holds:
		\begin{description}
			\item[ {\rm (a) }] If $\operatorname{span}(\phi_i)\cap C^o  = \emptyset$ for any eigenvector $\phi_i$, $i\in \{1,\ldots,n\}$ of $A$, then
			$$
				\overline{\lambda}_{C}(A)=	\max_{1\leq  j\leq n}\lambda_j(A),~~ \underline{\lambda}_{C}(A)=\min_{1\leq  j\leq n}\lambda_j(A)
		$$
			\item[{\rm (b) } ]  If there exists an  eigenvector $\phi_i$ of $A$ such that $\operatorname{span}(\phi_i)\cap C^o \neq \emptyset$, then $\overline{\lambda}_C(A)=\underline{\lambda}_C(A)=\lambda_i(A)$.
		\end{description}
	\end{cor}
	
	The following result takes advantage of the fact that the Perron–Frobenius theory in Theorem \ref{theorem1} extends to arbitrary matrices and cones.
	
\begin{thm}\label{theorem2}
	Assume that  $A \in M_{n\times n}(\mathbb{R})$, $C \in K(\mathbb{R}^n)$.
	\begin{align}
		 \mbox{If}~~v_C(A) \in C^o,~~\mbox{then}&\nonumber\\	&\overline{\lambda}_{C}(A+D)-\underline{\lambda}_{C}(A)\leq c_1(A,C)\|D\|_M, ~~\forall D \in M_{n\times n}(\mathbb{R}), \label{cont0}\\
		 \mbox{moreover, if}~ D\leq 0, ~\mbox{then}&\nonumber\\
		&\overline{\lambda}_{C}(A+D)\leq \underline{\lambda}_{C}(A), \label{cont0AD} \\
		\mbox{If}~~u_C(A) \in C^o,~~\mbox{then}&\nonumber\\ &\underline{\lambda}_{C}(A+D)-\overline{\lambda}_{C}(A)\geq -c_2(A,C)\|D\|_M, ~~\forall D \in M_{n\times n}(\mathbb{R}),  \label{cont00}\\
		\mbox{moreover, if}~ D\geq 0, ~\mbox{then}&\nonumber\\
		&\underline{\lambda}_{C}(A+D)\geq \overline{\lambda}_{C}(A),\label{cont00AD}
			\end{align}
	where 
	\begin{equation}
		c_1(A,C)=	\sup _{u \in C}\frac{\| u\| }{\langle u, v_{C}(A)\rangle},~~c_2(A,C)=	\sup _{v \in C}\frac{\| v\| }{\langle u_{C}(A), v \rangle}<+\infty.
	\end{equation}
In particular, if $u_C(A), v_C(A) \in C^o$, then 
	\begin{align}
		&|\overline{\lambda}_{C}(A+D)-\lambda_{C}(A)|\leq c_0(A,C) \|D\|_M, ~\forall D \in M_{n\times n}(\mathbb{R}), \label{contOC}\\
		&|\underline{\lambda}_{C}(A+D)-\lambda_{C}(A)|\leq c_0(A,C) \|D\|_M, ~\forall D \in M_{n\times n}(\mathbb{R}),\label{contOCO}
	\end{align}
	where $c_0(A,C):=\max\{c_1(A,C),c_2(A,C)\}$. 
\end{thm}

We say that a matrix $A$ has \textit{sign-constant elements off the diagonal}   if one of the following is satisfied: $a_{i,j}\geq 0$ or $a_{i,j}\leq 0$ for all $i,j \in \{1,\ldots,n\}$, $i\neq j$. In what follows, we denote by $M_{n\times n}^{isc}(\mathbb{R})$ the set irreducible matrix with  sign-constant elements off the diagonal.
		\begin{thm} \label{theorem3} 	
			
			Assume that $A \in M_{n\times n}^{isc}(\mathbb{R})$, then  $\lambda_{S_+}(A)$ is a simple eigenvalue of $A$,  and   $u_{S_+}, v_{S_+}
			$ are the corresponding positive right  and left eigenvectors. Moreover, $(u_{S_+}, v_{S_+})$ is a  saddle point of $\lambda(u,v)$ in $S_+$ which is unique up to the multipliers of $u_{S_+}$ and $v_{S_+}$. Furthermore,  the function $\lambda_{S_+}(\cdot)$ is continuous on $(M_{n\times n}(\mathbb{R}), \|\cdot\|_M)$ at any point $A \in M_{n\times n}^{isc}(\mathbb{R})$ in the following sense
			\begin{equation}\label{eq:ConC}
			\max\{|\underline{\lambda}_{S_+}(A+D)-\lambda_{S_+}(A)|,|\overline{\lambda}_{S_+}(A+D)-\lambda_{S_+}(A)|\}	\leq c_0(A,C) \|D\|_M, 
			\end{equation}
			$\forall D \in M_{n\times n}(\mathbb{R})$, where $c_0(A,C)$ does not depend on $D$,  moreover, $\lambda_{S_+}(A+D)\geq \lambda_{S_+}^*(A)$ if $D\geq 0$, and  $\overline{\lambda}_{S_+}(A+D)\leq \lambda_{S_+}^*(A)$ if $D\leq 0$.
		\end{thm}
		\begin{rem} The proof of  Theorems \ref{theorem1}, \ref{theorem3} provides a new approach to the proof of the Perron-Frobenius and Birkhoff-Varga Theorems. Indeed, from these results it follows that  if $A \in M_{n\times n}^{isc}(\mathbb{R})$  is a nonnegative matrix, then \eqref{eq:Noneg},  is fulfilled, while \eqref{eq:Re} holds true if $A \in M_{n\times n}^{isc}(\mathbb{R})$ is a  matrix with nonnegative elements off the diagonal.  Thus,  Theorem \ref{theorem3} and Corollary \ref{cor2} yield   the Perron–Frobenius and Birkhoff–Varga Theorems. Moreover, the theorem is supplemented with new properties of the essentially positive matrix such as continuity \eqref{eq:ConC}.
		\end{rem}
		\begin{rem} The  matrices with sign-constant elements off the diagonal  arise in a multitude of scientific disciplines and practical applications, see, e.g., \cite{Berman, Varga}. 
		\end{rem}

	 Note that inequalities \eqref{contOC}, \eqref{contOCO}, and \eqref{eq:ConC}  are related to Weyl’s inequality  \cite{parlet} for the eigenvalues of symmetric matrices, moreover they generalize it to arbitrary matrices. Furthermore,  \eqref{eq:ConC} means that the Perron root of irreducible nonnegative matrices is stable under arbitrary perturbations, which may be thought of as a generalization  of  Weyl's assertion on the stability of the spectrum of symmetric matrices under perturbations on the manifold of symmetric matrices \cite{parlet}.
		
			Clearly, $\lambda_{S_+}(A)$ in \eqref{eq:ConC} is the Peron root  if $A$ is an irreducible non-negative matrix.   
		But  $A+D$    may not even be the irreducible non-negative matrices, and thus, $\overline{\lambda}_{S_+}(A+D)$, $\underline{\lambda}_{S_+}(A+D)$ might be not the eigenvalue and  Peron root of $A+D$.
		In this regard, it appears that \eqref{eq:ConC} provides a definite  answer to the question posed by Meyer in \cite{Meyer}  about the continuity of the Peron root depending on the arbitrary matrices.
		
	Define  $d(C, C'):=\|I-U\|_M$ for $C,C'\in K(\mathbb{R}^n)$, $U \in O(n)$ such that $C'=UC$.  It is easily seen that  $d(C, C')$ is a well-defined metric on  $K(\mathbb{R}^n)$. The topology induced on $K(\mathbb{R}^n)$ due to this metric will be denoted by $\tau(O(n))$.

\begin{cor}\label{cor4}
		Assume that $A \in M_{n\times n}(\mathbb{R})$, $C \in K(\mathbb{R}^n)$  and $u_C(A), v_C(A) \in C^o$. Then the map $ \lambda_{(\cdot)}(A)$ is continuous on $(K(\mathbb{R}^n), \tau(O(n)))$ at $C$. Moreover, if $C' \in K(\mathbb{R}^n)$ such that $d(C, C')$ is sufficiently small, then 
		\begin{equation}\label{eq:contCone}
		\max\{|\overline{\lambda}_{C'}(A)- \lambda_{C}(A)|,~|\underline{\lambda}_{C'}(A)- \lambda_{C}(A)|\}	 \leq c_3(A,C)d(C, C'), 
		\end{equation}
		where $c_3(A,C)$ does not depends on $C'$.
	\end{cor}

		 Recall that a matrix $A$ is normal if $A^TA=A A^T$. It is well known that for any real normal matrix $A$ there exists  a set of invariant subspaces $V_j$, $1\leq j \leq n-l$   of $A$ with $l \in [0, n/2]$  such that $\operatorname{dim} V_j=2$,  $\operatorname{Im}\lambda(A)|_{V_j}\neq 0$, $j=1,\ldots l$, and  $\operatorname{dim} V_j=1$, $ \operatorname{Im}\lambda(A)|_{V_j}= 0$, $j=2l+1,\ldots,n$. Here $\lambda(A)|_{V_j}$ denotes the eigenvalue of the operator $A$ restricted on invariant subspaces $V_j$.
		 
		\begin{thm} \label{theorem44} 
			Assume that $A\in M_{n\times n}(\mathbb{R})$ and  $C\in K(\mathbb{R}^n)$.
		\begin{description}
			\item[$(1^o)$ ]	
			\begin{equation}\label{eq:norEst12}
				\min\{\lambda_j(\frac{A+A^T}{2})\mid~j=1,\ldots,n\}\leq  ~ \underline{\lambda}_C(A)\leq \overline{\lambda}_C(A)\leq 	\max\{\lambda_j(\frac{A+A^T}{2})\mid~j=1,\ldots,n\}.
					\end{equation}
				\item[ $(2^o)$] 	If $A$ is a normal matrix,  then 
			\begin{equation*}
				\min_{1\leq  j\leq n}\operatorname{Re}\lambda_j(A) \leq ~ \underline{\lambda}_C(A)\leq \overline{\lambda}_C(A) ~\leq 
				\max_{1\leq  j\leq n}\operatorname{Re}\lambda_j(A), ~~\forall C \in K(\mathbb{R}^n).
			\end{equation*}

		\end{description}
	\end{thm}
		Note that  the eigenvalues of the symmetric matrix $(A+A^T)/2$ are real.
		
	\begin{cor}
		If $A$ is a skew-symmetric matrix, then $\underline{\lambda}_C(A)= \overline{\lambda}_C(A)=0$ for any $C \in K(\mathbb{R}^n)$.
	\end{cor}
	
	\begin{thm} \label{theorem4} Assume that $A$ is a real normal matrix,  $C\in K_O(\mathbb{R}^n)$.  Then there holds the following:
	\begin{description}
		\item[ $(1^o)$] If $V_i\cap C^o = \emptyset$, $\forall i=1,\ldots n-l$, then
		\begin{equation*}
			 \overline{\lambda}_{C}(A)=	\max_{1\leq  j\leq n}\operatorname{Re}\lambda_j(A),~~ \underline{\lambda}_{C}(A)=\min_{1\leq  j\leq n}\operatorname{Re}\lambda_j(A).
		\end{equation*}	
		\item[$(2^o)$ ]  If  $\exists i\in \{1,\ldots,n\}$ such that $V_i\cap C^o \neq \emptyset$, then 
		$$
		\overline{\lambda}_C(A)=\underline{\lambda}_C(A)=\operatorname{Re}\lambda_i(A).
		$$
		 \end{description}
	\end{thm}

\begin{rem}\label{rem: Fin}
The Birkhoff-Varga formula \eqref{BV} and its generalizations \eqref{eq:DefQQ}, \eqref{eq:DefQQ2}, being variational, allows us to apply variational methods to numerically finding quasi-eigenpairs of matrices, such as gradient-based algorithms and other non-iterative methods, see, e.g.,  \cite{Bagirov, IvanIlya, IlIvan1, Salazar}. Thus, in particular cases, such as irreducible nonnegative matrices, variational formulas \eqref{eq:DefQQ}, \eqref{eq:DefQQ2} can produce algorithms for finding eigenpairs based on new principles, which may have certain advantages (see, e.g., \cite{ IvanIlya, IlIvan1}) over the commonly used ones, see, e.g, \cite{golub, Kuznet} and references therein.
\end{rem}

\begin{rem} 
	We believe that the generalization of the Birkhoff-Varga formula  \eqref{IG}   may be developed for other spectral problems in matrix theory, such as nonlinear spectral problems \cite{Mehrmann, Voss}, spectral problems for multilinear forms \cite{Chang, FriedlandHan} etc. \cite{ Weinberger}.
\end{rem}

	\begin{rem}
		Identifying the largest real part of eigenvalues of matrices with non-zero imaginary parts is important in finding solutions to various problems, see, e.g., \cite{Arnold, chen, Hassard, Hale, Krasnoselskii, Varga}. For instance, verifying that $ \displaystyle{\max_{1\leq  j\leq n}\operatorname{Re}\lambda_j(A)\geq 0}$ is crucial in the detection of the Hopf bifurcation, see, e.g., \cite{chen, Hassard, Hale}. Meanwhile, existing methods of verifying this condition, such as the Routh-Hurwitz criterion, see, e.g., \cite{Hale}, are often difficult to apply, especially for large matrices  see, e.g., \cite{chen}. An alternative strategy is provided by Corollary \ref{cor2}, Theorems \ref{theorem44}, \ref{theorem4}  for addressing this problem, and it can be anticipated that its further development will produce the desired outcome. 
	\end{rem}

\section{ Proof of Theorem \ref{theorem1}}
A function 
$f:X\to \mathbb{R}$  defined on a convex subset 
$X$ of a real vector space is \textit{quasiconvex} if for all 
$x,y\in X$ and 
$\alpha \in [0,1]$ we have
$$
f(\alpha x+(1-\alpha )y)\leq \max {\big \{}f(x),f(y){\big \}}.
$$
A function whose negative is quasiconvex is called \textit{quasiconcave}. A function which is quasiconvex and quasiconcave said to be \textit{quasilinear} (see \cite{Sion}).

Our method is based on  Sion's minimax theorem \cite{Sion}:

\noindent
{\bf Theorem} (\textit{Sion})
	\textit{Let $X$ be a compact convex subset of a linear topological space and 
	$Y$ a convex subset of a linear topological space. If 
	$f$ is a real-valued function on 
	$X\times Y$  with}
	\begin{itemize}
		\item \textit{$f(x,\cdot)$ upper semicontinuous and quasi-concave on 
		$Y$, $\forall x\in X$, and}
		\item $f(\cdot,y)$\textit{ lower semicontinuous and quasi-convex on }
		$X$, $\forall y\in Y$,
	\end{itemize}
	\textit{then the minimax principle is satisfied}
	$$
	\sup_{y\in Y}\min_{x\in X}f(x,y)=\min_{x\in X}\sup_{y\in Y} f(x,y).
	$$

Let us prove	
	\begin{prop}\label{propSion} Let $A \in M_{n\times n}(\mathbb{R})$, $C \in K(\mathbb{R}^n)$. 
		Then the functions  $\lambda(\cdot, v)$, $\forall v \in C$ and $\lambda(u,\cdot)$,	$\forall u \in C$ are  quasilinear on $C^o$.
	\end{prop}
	\begin{proof} Let $v \in C^o$, $u,w \in C$. Consider
			\begin{align*}
			f(\alpha):=&\lambda(\alpha u+(1-\alpha)w,v)=\frac{\alpha\left\langle A u, v\right\rangle+(1-\alpha)\left\langle A w, v\right\rangle}{\left\langle \alpha  u+(1-\alpha) w ,v\right\rangle},~~\alpha \in [0,1].
	\end{align*}
	Calculate
	\begin{align*}
		f'(\alpha)=	\frac{\left\langle A w, v\right\rangle \left\langle  u, v\right\rangle-\left\langle A u, v\right\rangle\left\langle w, v\right\rangle}{\left\langle \alpha  u+(1-\alpha) w ,v\right\rangle^2},~~\alpha \in (0,1).
	\end{align*}
	Hence if $\left\langle A w, v\right\rangle \left\langle  u, v\right\rangle-\left\langle A u, v\right\rangle\left\langle w, v\right\rangle \neq 0$, then $f'(\alpha)> 0$ or $f'(\alpha)< 0$,  $\forall \alpha \in (0,1)$, whereas
				if $\left\langle A w, v\right\rangle \left\langle  u, v\right\rangle-\left\langle A u, v\right\rangle\left\langle w, v\right\rangle = 0$, then $f'(\alpha)\equiv 0$, $\forall \alpha \in (0,1)$.
	This implies that
		$$
\min\{\lambda(u,v),\lambda(w,v)\}\leq \lambda(\alpha u+(1-\alpha)w,v)\leq \max\{\lambda(u,v),\lambda(w,v)\}, ~~\alpha \in (0,1).
$$
Hence, $\lambda(\cdot, v)$ is quasiconvex and quasiconcave function, and therefore, is a quasilinear on $C$.
	The proof for $\lambda(u, \cdot)$ is similar.
\end{proof}

We need also
\begin{prop}\label{prop2}
	$\lambda_{inf}(u):=\inf_{v \in C^o}\lambda(u,v)$ is a upper semicontinuous functional in $C$.
\end{prop}
\begin{proof} Let $(\phi_n)_{n=1}^\infty$ be a 
	countable dense set in $C^o$. Clearly, 
	$$
	\lambda_{inf}(u)=\inf_{n\ge 1}\lambda(u,\phi_n), ~~u \in C.
	$$
	Hence,	for any ${\displaystyle \tau \in \mathbb{R} }$, 
	$$
	\{u \in C \mid \lambda_{inf}(u)< \tau\}=\bigcup_{n=1}^{\infty}\{u \in C \mid \lambda(u,\phi_n)< \tau\}.
	$$
	It is easily seen that  $\lambda(\cdot,\phi_n)$ is a continuous function in $C$, for $n=1,\ldots $. Hence the set $\{u \in C \mid \lambda(u,\phi_n)< \tau\}$ is open in $C$, and therefore, $\{u \in C\mid \lambda_{inf}(u)< \tau\}$ is open for any ${\displaystyle \tau \in \mathbb{R} }$. This means that $\lambda_{inf}(u)$ is  a upper semicontinuous functional.
	
\end{proof}

{\it \underline{Proof of Theorem \ref{theorem1}}}

 Clearly,   $\sup_{u\in C}\inf_{v \in C^o}\lambda(u,v)>-\infty$ and 	$\inf_{v \in C^o}\sup_{u\in C}\lambda(u,v)<+\infty$. Since there holds $\sup_{u\in C}\inf_{v \in C^o}\lambda(u,v)\leq \inf_{v \in C^o}\sup_{u\in C}\lambda(u,v)$, we have $$
 -\infty<\sup_{u\in C}\inf_{v \in C^o}\lambda(u,v)\leq \inf_{v \in C^o}\sup_{u\in C}\lambda(u,v)<+\infty.
 $$

	Define $B_1=\{x \in C:~||x||\leq 1\}$. Clearly, $\partial \dot{B}_1:=\{x \in C:~||x||=1\}$ is a compact in $\mathbb{R}^n$.
	By  homogeneity of  $\lambda(\cdot,v)$ in $C$, we have
	$\overline{\lambda}_C(A):=\sup_{u\in C}\inf_{v \in C^o}\lambda(u,v)=\sup_{u\in \partial \dot{B}_1}\inf_{v \in C^o}\lambda(u,v)$.   By Proposition \ref{prop2},
	$\lambda_{inf}(u):=\inf_{v \in C^o}\lambda(u,v)$ is  upper semicontinuous on $ B_1$,  and therefore, there exists $u_{C} \in \partial \dot{B}_1$ such that $\overline{\lambda}_{C}(A)=\inf_{v \in C^o}\lambda(u_{C},v)=\sup_{u\in C}\inf_{v \in C^o}\lambda(u,v)$.

It is easily found  a sequence  of compact convex subsets $ B_1(\epsilon) \subset B_1$, $\epsilon\in (0,1)$ such that $  B_1(\epsilon) \subset B_1(\epsilon')$, $\forall \epsilon\geq \epsilon'$ and $\cup_{\epsilon>0}  B_1(\epsilon) =B_1$. Since $\lambda(\cdot,\cdot) \in C( B_1(\epsilon)\times C^o)$,  Proposition \ref{propSion} and  Sion's theorem yield 
		$$
		\sup_{u\in B_1(\epsilon)}\inf_{v \in C^o}\lambda(u,v)=
		\inf_{v \in C^o}\sup_{u\in B_1(\epsilon)}\lambda(u,v),~~\epsilon\in (0,1).
		$$
Using $\lambda(\cdot,\cdot) \in C( B_1(\epsilon)\times C^o)$ it is not hard to show that
\begin{equation}\label{eq:converg1}
	\inf_{v \in C^o}\sup_{u\in B_1(\epsilon)}\lambda(u,v)	\to \inf_{v \in C^o}\sup_{u\in B_1}\lambda(u,v),~~ \mbox{as}~~\epsilon \to 0. 
\end{equation}
 Since $\cup_{\epsilon>0}  B_1(\epsilon) =B_1$,  there exists $\epsilon_0>0$ such that $u_{C} \in B_1(\epsilon)$ for any $\epsilon<\epsilon_0$. Hence, 
		$$
		\sup_{u\in B_1(\epsilon)}\inf_{v \in C^o}\lambda(u,v)=\sup_{u\in B_1}\inf_{v \in C^o}\lambda(u,v)=\overline{\lambda}_C(A), ~~\forall \epsilon \in (0,\epsilon_0),
		$$
and consequently, by \eqref{eq:converg1} we get  the minimax principle for $\lambda(u,v)$  in $C\times C^o$
	$$
	\sup_{u\in C}\inf_{v \in C^o}\lambda(u,v)=
	\inf_{v \in C^o}\sup_{u\in C}\lambda(u,v).
	$$
Similar arguments apply for proving the existence of the left quasi-eigenvector $v_C(A)\in C$ and that the minimax principle for $\lambda(u,v)$ in  $C^o\times C$ holds.

Assume that $u_C(A) \in C^o$. This implies $\overline{\lambda}_C(A)=\sup_{u\in C}\inf_{v \in C^o}\lambda(u,v)=\sup_{u\in C^o}\inf_{v \in C^o}\lambda(u,v)$. Hence, we have
\begin{align*}
	\overline{\lambda}_C(A)=\inf_{v \in C^o}\sup_{u \in C}\lambda(u,v)\geq \inf_{v \in C^o}\sup_{u \in C^o}\lambda(u,v)\geq \sup_{u\in C^o}\inf_{v \in C^o}\lambda(u,v)=\overline{\lambda}_C(A), 
\end{align*} 
and thus,
$$
\overline{\lambda}_C(A)= \inf_{v \in C^o}\sup_{u \in C^o}\lambda(u,v)= \sup_{u\in C^o}\inf_{v \in C^o}\lambda(u,v), 
$$
that is, the minimax principle  in  $C^o\times C^o$ holds true. The same reasoning applies to  the case $v_C(A) \in C^o$.

We thus have, if $u_C(A), v_C(A) \in C^o$, then $\lambda_C(A):=\overline{\lambda}_C(A)=\underline{\lambda}_C(A)$, and
\begin{align*}
\lambda_C(A)=\inf_{v \in C^o}\lambda(u_C(A),v)\leq \lambda(u_C(A),v_C(A)),\\
\lambda_C(A)=\sup_{u\in C^o}\lambda(u,v_C(A))\geq \lambda(u_C(A),v_C(A)).
\end{align*}
Thus, $\lambda_C(A)=\lambda(u_C(A),v_C(A))=\inf_{v \in C^o}\lambda(u_C(A),v)=\sup_{u\in C^o}\lambda(u,v_C(A))$. Since  $u_C(A), v_C(A)$ are internal points in open set $C^o$, this implies $\partial_u \lambda(u_C(A),v_C(A))=0$, $\partial_v \lambda(u_C(A),v_C(A))=0$, and thus, $u_C(A)$, $v_C(A)$ are right and left eigenvectors of $A$ with eigenvalue $\lambda_C(A)$.

{\it \underline{Proof of Corollary \ref{cor1}}}.\, 

\textit{(i)}  By Theorem \ref{theorem1} there exists a left   quasi-eigenvectors $v_C(A) \in C$ of $A$. Hence, since $\operatorname{span}(\phi_i )\cap C^o \neq \emptyset$,
$$
\underline{\lambda}_{C}(A) =\sup _{u \in C^o}\frac{\left\langle A u, v_C(A)\right\rangle}{\left\langle u, v_C(A)\right\rangle}\geq \frac{\left\langle A \phi_i, v_C(A)\right\rangle}{\left\langle \phi_i, v_C(A)\right\rangle}=\lambda_i(A).
$$
The case $\operatorname{span}(\psi_i )\cap C^o \neq \emptyset$ is handled by the similar method. 

\textit{(ii)} If $\operatorname{span}(\phi_i )\cap C^o \neq \emptyset,\operatorname{span}(\psi_i )\cap C^o \neq \emptyset$, then by the above $\overline{\lambda}_C(A)\leq \lambda_i(A)\leq \underline{\lambda}_C(A)$. Since $\overline{\lambda}_C(A)\geq \underline{\lambda}_C(A)$, we derive $\overline{\lambda}_C(A)= \lambda_i(A)=\underline{\lambda}_C(A)=\lambda(\phi_i,\psi_i) $. Hence, using \eqref{eq:DefQQ}, \eqref{eq:DefQQ2} we infer that $u_C(A)=\phi_i$, $v_C(A)=\psi_i$ up to multipliers. The converse statement can be proved similar.

{\it \underline{Proof of Corollary \ref{cor2}}}\,

\textit{(i)}\, Let $\lambda_i(A) \in \mathbb{C}$, $i=1,\ldots,n$ be an eigenvalue of $A$. Suppose $\phi_i \in  \mathbb{C}^n$ is a corresponding right eigenvector, i.e., $A\phi_i=\lambda_i(A) \phi_i$. Since $A\geq 0$, $|\lambda_i(A)||\phi_i|=|\lambda_i(A) \phi_i|\leq A|\phi_i|$, and consequently, 
$$
\lambda(|\phi_i|,v)= \frac{\langle A |\phi_i|, v\rangle}{\langle |\phi_i|, v\rangle} \geq |\lambda_i(A)|, ~~\forall v \in S_+^o,~~i=1,\ldots,n. 
$$
Hence by \eqref{eq:DefQQ} we have
$$
\overline{\lambda}_{S_+}(A)=\sup _{u\in S_+}\inf _{v\in S_+^o}\frac{\left\langle A u, v\right\rangle}{\left\langle u, v\right\rangle}\geq \inf _{v\in S_+^o} \frac{\langle A |\phi_i|, v\rangle}{\langle |\phi_i|, v\rangle}\geq  |\lambda_i(A)|
$$
$\forall i=1,\ldots,n$, and thus, we obtain \eqref{eq:Noneg1}. Assume that $\overline{\lambda}_{S_+}(A)$ is an eigenvalue. Then by \eqref{eq:Noneg1}, $\overline{\lambda}_{S_+}(A)>0$, and therefore, $\overline{\lambda}_{S_+}(A) \in \{|\lambda_j(A)|,~j=1,\ldots,n\}$. Hence  \eqref{eq:Noneg} holds true.

\textit{(ii)}\,   Clearly, there exists a sufficiently large $\gamma>0$ such that $A+\gamma I$ is a nonnegative matrix, where  $I$ is an identity matrix. Then by \textit{(i)}, Corollary \ref{cor2},  we obtain $\overline{\lambda}_{S_+}(A)+\gamma\geq |\lambda_i(A)+\gamma|\geq \operatorname{Re}(\lambda_i(A)) +\gamma$, $\forall i=1,\ldots,n$, and consequently, $\overline{\lambda}_{S_+}(A)\geq  \operatorname{Re}(\lambda_i(A)) $, $i=1,\ldots,n$, which yields \eqref{eq:Re1}. The proof of \eqref{eq:Re} is similar to \eqref{eq:Noneg}.

{\it \underline{Proof of Corollary \ref{cortheorem4}}}\, 
 Given \eqref{inv}-\eqref{inv2}, it is sufficient to prove the assertions of the corollary for diagonal matrix $A= \operatorname{diag}(\lambda_1, ...,\lambda_n)$. In this case, $\phi_i= e_i$, $i\in \{1,\ldots,n\}$ and we may assume that $\lambda_1\leq  ...\leq \lambda_n$. Let $C\in K_O(\mathbb{R}^n)$. It is easily seen that  the only following is possible: 1) $\operatorname{span}(e_i)\cap \partial C  \neq \emptyset$,  $\forall i\in \{1,\ldots,n\}$; 2) $ \exists e_i$ such that $\operatorname{span}(e_i)\cap C^o \neq \emptyset$. Note that  cases 1) and 2) imply assumptions (a) and (b) of the corollary are satisfied, respectively. Moreover, case 1) means that $C$ coincides with some orthant of $\mathbb{R}^n$.

Assume that 1) is satisfied.  Since $u_iv_i\geq 0$, $i=1,\ldots,n$, for any $u:=(u_1\ldots u_n), v_n:=(v_1\ldots v_n)$ belonging the orthant $C$, we have
$$
\overline{\lambda}_{C}(A)=\sup _{u \in C}\inf _{v\in C^o}\frac{\left\langle A u, v\right\rangle}{\left\langle u, v\right\rangle}\leq \lambda_n+\sup _{u \in C}\inf _{v\in C^o}\frac{ \sum_{i=1}^n (\lambda_i-\lambda_n) u_i v_i}{\left\langle u, v\right\rangle}\leq \lambda_n.
$$
 On the other hand,
$$
\overline{\lambda}_{C}(A)=\sup _{u \in C}\inf _{v\in C^o}\frac{\left\langle A u, v\right\rangle}{\left\langle u, v\right\rangle}\geq \inf _{v\in C^o}\frac{\left\langle A e_n, v\right\rangle}{\left\langle e_n, v\right\rangle}= \lambda_n. 
$$
Hence, we get $\overline{\lambda}_{C}(A)= \lambda_n=	\max_{1\leq  j\leq n}\lambda_j(A)$. By a similar argument, we have $\underline{\lambda}_{C}(A)=\min_{1\leq  j\leq n}\lambda_j(A)$. 

Assuming that 2) is true, assertion (b) follows directly from Corollary \ref{cor1}.

\section{Proofs of Theorems \ref{theorem2}}

Since $v_C(A) \in C^o$,  $\inf_{u\in\partial \dot{B}_1}\langle u, v_{C}(A)\rangle  >0$, where $\partial \dot{B}_1=\{x \in C:~||x||=1\}$. Thus, $\lambda(\cdot, v_C(A))$ is a continuous bounded function on $C$, and therefore, we have  $\underline{\lambda}_{C}(A)=\sup _{u \in C^o}\lambda(u, v_C(A))=\sup _{u \in C}\lambda(u, v_C(A))$. From this and 
by Theorem \ref{theorem1} we derive 
\begin{align}
	\overline{\lambda}_{C}(A+D)= & \sup _{u \in C}\inf _{v\in C^o}\frac{\langle (A+D) u, v\rangle}{\langle u, v\rangle}\leq \sup _{u \in C}\frac{\langle (A+D) u, v_{C}(A)\rangle}{\langle u, v_{C}(A)\rangle}\leq \nonumber \\ 
	&\sup _{u \in C}\frac{\langle A u, v_{C}(A)\rangle}{\langle u, v_{C}(A)\rangle}+\sup _{u \in C}\frac{\langle D u, v_{C}(A)\rangle}{\langle u, v_{C}(A)\rangle}=\underline{\lambda}_{C}(A)+\sup _{u \in C^o}\frac{\langle D u, v_{C}(A)\rangle}{\langle u, v_{C}(A)\rangle}.\label{eq:IneqF}
\end{align}
Note that
$$
\sup _{u \in C^o}\frac{\langle D u, v_{C}(A)\rangle}{\langle u, v_{C}(A)\rangle}\leq \|D\|_M \|v_{C}(A)\|\sup _{u \in C}\frac{\| u\| }{\langle u, v_{C}(A)\rangle}.
$$
Since  $\inf_{u\in \partial \dot{B}_1}\langle u, v_{C}(A)\rangle  >0$, by 
 the homogeneity of  $\frac{\| \cdot\| }{\langle \cdot, v_{C}(A)\rangle}$  we have
$$
0<\sup _{u \in C^o}\frac{\| u\| }{\langle u, v_{C}(A)\rangle}=c_1(A,C)<+\infty,
$$
Thus we have proved \eqref{cont0}. Clearly, if $D\leq 0$, then \eqref{eq:IneqF} implies \eqref{cont0AD}. 
Proof of \eqref{cont00}, \eqref{cont00AD} are similar.

Let $u_C(A), v_C(A) \in C^o$. Then by Theorem \ref{theorem1}, $\lambda_{C}(A)=\overline{\lambda}_{C}(A)=\underline{\lambda}_{C}(A)$.  Since $\overline{\lambda}_{C}(A+D)\geq \underline{\lambda}_{C}(A+D)$,  \eqref{cont0}, \eqref{cont00} imply
\begin{align*}
	-c_2(A,C)\|D\|_M\leq \underline{\lambda}_{C}(A+D)-\lambda_{C}(A)\leq \overline{\lambda}_{C}(A+D)-\lambda_{C}(A)\leq c_1(A,C)\|D\|_M,
\end{align*}
$\forall D \in M_{n\times n}(\mathbb{R})$. Thus, we get \eqref{contOC}, \eqref{contOCO}.

\section{Proofs of Theorem \ref{theorem3}}

	To be specific, consider the case $a_{ij}\leq 0$, $\forall i\neq j$. In view of that
	$A -\overline{\lambda}_{S_+}(A) I \in M_{n\times n}^{isc}(\mathbb{R})$, the inequality $ A u_{S_+}-\overline{\lambda}_{S_+}(A) u_{S_+}\geq 0 $,  for $u_{S_+} \in S_+$ implies that $u_{S_+} \in S_+^o$. Therefore, $\langle u_{S_+},  v_{S_+} \rangle \neq 0$, and thus, $\lambda(u_{S_+}, v_{S_+})$ is well-defined. Hence, by Theorem \ref{theorem1}
	$$
	\overline{\lambda}_{S_+}(A)=\lambda_{S_+}(A)=\sup_{u \in S_+^o}\inf_{v \in S_+}\lambda(u,v) =\inf_{v \in S_+}\lambda(u_{S_+} ,v)\leq \lambda(u_{S_+}, v_{S_+}).
	$$
	On the other hand,
	$$
\overline{\lambda}_{S_+}(A) \geq	\underline{\lambda}_{S_+}(A)=\sup_{u\in S_+^o}\lambda(u,v_{S_+})\geq  \lambda(u_{S_+}, v_{S_+})
	$$
	Thus,
		\begin{equation}\label{eq:irre5}
		\lambda_{S_+}(A)=\underline{\lambda}_{S_+}(A)=\inf_{v \in S_+}\lambda(u_{S_+} ,v)=\sup_{u\in S_+}\lambda(u,v_{S_+})= \lambda(u_{S_+}, v_{S_+}).
	\end{equation}
Hence and since $u_{S_+} \in  S_+^o$, we have $\partial_u\lambda(u_{S_+},v_{S_+})= 0$, i.e., $A^Tv_{S_+}= \lambda_{S_+}(A) v_{S_+}$.
	 Since $A^T -\lambda_{S_+}(A) I \in M_{n\times n}^{isc}(\mathbb{R})$, equality $ A^T v_{S_+}-\lambda_{S_+}(A) v_{S_+}= 0$ can only be achieved if $v_{S_+} \in S_+^o$, and hence, by \eqref{eq:irre5} 
	we get  $\partial_v\lambda(u_{S_+},v_{S_+})=0$. Thus, the first part of the theorem is  proved.

The uniqueness of  $u_{S_+}(A)$ and $ v_{S_+}(A)
$ up to multipliers is a consequence of the Perron–Frobenius theorem. Indeed, since by the assumption  $A$ is  an irreducible matrix with  sign-constant elements off the diagonal, one can find a sufficiently large $\gamma>0$ such that $(A+\gamma I)$ or $(-A+\gamma I)$ are irreducible matrices with non-negative elements. Hence  by the Perron–Frobenius theorem $u_{S_+}(A)$ and $ v_{S_+}(A)
$ are unique up to multipliers of the right and left eigenvectors of $(A+\gamma I)$ or $(-A+\gamma I)$. 

Since $u_{S_+}(A), v_{S_+}(A) \in S_+^o$, Theorem \ref{theorem2} yields \eqref{eq:ConC}.

{\it Proof of Corollary \ref{cor4}}.\,

	Let  $C'=UC$ with $U \in O(n)$. Observe
	\begin{equation*}
		\overline{\lambda}_{C'}(A)=\overline{\lambda}_{UC}(A)=\sup _{u\in C}\inf _{v\in C^o}{\frac{\left\langle U^TA U u, v\right\rangle}{\left\langle u, v\right\rangle}}=\overline{\lambda}_{C}( U^TA U).
	\end{equation*}	
Since $u_C(A), v_C(A) \in C^o$,  Theorem \ref{theorem2} yields 
	\begin{equation}\label{eq: 4.2}
		|\overline{\lambda}_{C'}(A)- \lambda_{C}(A)|=|\overline{\lambda}_{C}( U^TA U)- \lambda_{C}(A)|\leq c_0(A,C)\| U^TA U-A\|.
	\end{equation}
	By Stone's theorem there exists a self-adjoint matrix $G$ such that  $U=\exp(G)$, and therefore,  $U=I+G+\bar{o}(\|G\|)$ for sufficiently small $\|G\|$, where $\bar{o}(\|G\|)$ is a little-o of $\|G\|$ as $\|G\| \to 0$. From this $ U^TA U=A+[A,G]+\bar{o}(\|G\|)$ with $[A,G]:=AG-GA$, and thus 
	\begin{equation}\label{eq:UTU}
		\| U^TA U-A\| =\|[A,G]+\bar{o}(\|G\|)\|\leq \|G\|\|[A,\tilde{G}]+\bar{o}(1)\|\leq c_3(A,C) \|G\|
	\end{equation}
	for sufficiently small $\|G\|$. Here  $\tilde{G}=G/\|G\|$, and a constant $c_3(A,C) \in (0,+\infty)$ does not depends on $G$. Since $d(C, C')=\|U-I\|=\|G+\bar{o}(\|G\|)\|$, we have $\|G\| \to 0$ as $d(C, C') \to 0$. This by \eqref{eq: 4.2}, \eqref{eq:UTU} implies
	$$
	|\overline{\lambda}_{C'}(A)- \lambda_{C}(A)|	 \leq c_3(A,C)d(C, C').
	$$
	The rest of \eqref{eq:contCone} is obtained similarly.
	
\section{Proof of Theorem \ref{theorem44}}
	
First we prove	$(2^o)$. It is easily follows from  \eqref{eq:DefQQ} 
	\begin{align}\label{ss1}
		&\overline{\lambda}_C(A)=\sup _{u \in C}\inf _{v\in C^o}\frac{\left\langle A u, v\right\rangle}{\left\langle u, v\right\rangle}\leq \sup _{u \in \mathbb{R}^n\setminus 0}\frac{\left\langle A u, u\right\rangle}{\left\langle u, u\right\rangle}=\sup _{x\in \mathbb{C}^n\setminus 0: \operatorname{Im} x=0}\frac{\langle\langle A x, x\rangle\rangle}{\langle\langle x, x\rangle\rangle}.
	\end{align}
	 Here  $\langle\langle x,y\rangle\rangle:=\sum_{i=1}^n x_i \overline{y}_i$, $x,y \in \mathbb{C}^n$.
	The set  $J:=\{z =\langle\langle Ax,x\rangle\rangle\mid x \in \mathbb{C}^n, ~\langle\langle x, x\rangle\rangle = 1\}$ is called  the field of values of $A$ (see \cite{marcus}).  It is known that for the normal matrix $A$, $J$  coincides with the convex hull $H(\lambda(A))$ of the set of eigenvalues $\{\lambda_i(A),~i=1,\ldots,n\}$ (see \cite{marcus}, p. 168). 
	This by \eqref{ss1} implies that  
	$$
	\overline{\lambda}_C(A)\leq \sup \operatorname{Re}(H(\lambda(A))) =\max\{\operatorname{Re}(\lambda_i(A)),~i=1,\ldots,n\}.
	$$
	Similarly,  from  \eqref{eq:DefQQ} and by the minimax principle \eqref{eq:mimaPrin}  in $C\times C^o$ we have
	\begin{align}\label{ss2}
		&\overline{\lambda}_C(A)=\inf _{v\in C^o}\sup _{u \in C}\frac{\left\langle A u, v\right\rangle}{\left\langle u, v\right\rangle}\geq \inf _{v  \in \mathbb{R}^n\setminus 0}\frac{\left\langle A v, v\right\rangle}{\left\langle v, v\right\rangle},
	\end{align}
	and therefore $
	\overline{\lambda}_C(A)\geq \inf \operatorname{Re}H(\lambda(A)) =\min\{\operatorname{Re}\lambda_i(A),~i=1,\ldots,n\}$.
	In the same way, using the minimax principle \eqref{eq:mimaPrin}  in $C^o\times C$ we derive
	$$
	\min\{\operatorname{Re}\lambda_i(A),~i=1,\ldots,n\}\leq \underline{\lambda}_C(A) \leq \max\{\operatorname{Re}(\lambda_i(A)),~i=1,\ldots,n\}
	$$

Let us now prove	$(1^o)$.  Observe,
	$$
	\left\langle\frac{A+A^T}{2}u,u \right\rangle =\left\langle A u, u\right\rangle,~~\forall u \in \mathbb{R}^n.
	$$
	Hence, by \eqref{ss1}, \eqref{ss2}
	$$
	\inf _{u \in C^o}\frac{\left\langle (A+A^T) u, u\right\rangle}{2\left\langle u, u\right\rangle}\leq \overline{\lambda}_C(A)\leq \sup _{u \in C}\frac{\left\langle (A+A^T) u, u\right\rangle}{2\left\langle u, u\right\rangle}.
	$$
	Note that $A+A^T$ is a symmetric, and consequently normal matrix. Hence, the same arguments as in the above proof of  $(1^o)$  implies that
	$$
	\min\{\lambda_i(\frac{A+A^T}{2})\mid~i=1,\ldots,n\}\leq \overline{\lambda}_{C}(A)\leq 	\max\{\lambda_i(\frac{A+A^T}{2})\mid~i=1,\ldots,n\},
	$$
	where we  take into account that the eigenvalues of the symmetric matrix $A+A^T$ are real.
	The inequality for $\underline{\lambda}_{C}(A)$ in \eqref{eq:norEst12} is obtained similarly.
	\qed

\section{Proof of Theorem \ref{theorem4}}

It is well known (see, e.g., \cite{marcus}) that for any normal  matrix $A$ there corresponds an  orthogonal matrix $U_A$ that reduces $A$ to the canonical form, i.e.,
\begin{equation}\label{eq:canNorm}
	{\displaystyle U^T_AA U_A= \mathcal{O}_A={\begin{bmatrix}{\begin{matrix}Q(r_1, \theta_1)&&\\&\ddots &\\&&Q(r_l, \theta_l)\end{matrix}}&0\\0&{\begin{matrix}\mu_{2l+1}&&\\&\ddots &\\&&\mu_{n}\end{matrix}}\\\end{bmatrix}},}
\end{equation}
where  $\theta_i \in [0,2\pi)$, $r_i \in \mathbb{R}$, $i=1,\ldots l$, $0\leq l\leq n/2$, $\mu_i \in \mathbb{R}$, $i=2l+1,\ldots,n$, and the matrix $Q(r,\theta)$ is defined as follows
$$
Q(r,\theta):=
r \left[ {\begin{array}{cc}
		\cos\,\theta & -\sin\, \theta \\
		\sin\,\theta & \cos\,\theta \\
\end{array} } \right], ~~\theta \in [0,2\pi), ~~r \in \mathbb{R}.
$$
 Note that $\lambda_{1}(Q(r,\theta))= r(\cos\, \theta +i \sin\,\theta)$, $\lambda_2(Q(r,\theta))= r(\cos\,\theta -i \sin\,\theta)$, and thus, $\operatorname{Re} \lambda_{j}(Q(r,\theta))=r\cos\, \theta$, $j=1,2$.

In light of \eqref{eq:canNorm} and  \eqref{inv}-\eqref{inv2}, it is sufficient to prove the assertions of the theorem for the  matrix in the canonical form $\mathcal{O}_A$ with $C\in K_0(\mathbb{R}^n)$. Without loss of generality, we may assume that $l=n/2$. Note that in this case we have $V_i=\operatorname{span}(\{e_i,e_{i+1}\})$, $i=1,\ldots, l$.

Let $C\in K_O(\mathbb{R}^n)$. It is easily seen that  the only following is possible: 1) $V_i\cap \partial C  \neq \emptyset$,  $\forall i\in \{1,\ldots,n\}$, or 2) $ \exists V_i$ such that $V_i\cap C^o \neq \emptyset$. Note that  1) and 2) imply assumptions of $(1^o)$ and $(2^o)$ of the theorem, respectively.  Additionally, 1) implies that $C$ coincides with some orthant, which without losing generality may be assumed to be a nonnegative orthant of $\mathbb{R}^n$.

Assume that 1) is satisfied. Observe,
$$
\overline{\lambda}_{C}(\mathcal{O}_A)=\sup _{u \in C}\inf _{v\in C^o}\frac{\left\langle \mathcal{O}_A u, v\right\rangle}{\left\langle u, v\right\rangle}\geq \sup _{u \in V_i}\inf _{v\in C^o}\frac{\left\langle \mathcal{O}_A u, v\right\rangle}{\left\langle u, v\right\rangle}\geq  \operatorname{Re}\lambda_i(\mathcal{O}_A),~i=1,\ldots,l, 
$$
and thus, $\overline{\lambda}_{C}(\mathcal{O}_A)\geq \max_{1\leq  j\leq n}\operatorname{Re}\lambda_j(A)$. 
On the other hand, by Theorem \ref{theorem44}, we have  $\overline{\lambda}_{C}(\mathcal{O}_A)\leq \max_{1\leq  j\leq n}\operatorname{Re}\lambda_j(\mathcal{O}_A)$. Hence, we get  $\overline{\lambda}_{C}(\mathcal{O}_A)= \max_{1\leq  j\leq n}\operatorname{Re}\lambda_j(\mathcal{O}_A)$.
By a similar argument,  $\underline{\lambda}_{C}(\mathcal{O}_A)=\min_{1\leq  j\leq n}\lambda_j(\mathcal{O}_A)$. 

Assume that 2) is true.  Without loss of generality we may assume that $i=1$, i.e.,  $V_1\cap C^o \neq \emptyset$, and therefore, $\operatorname{span}(e_1)\cap C^o \neq \emptyset$ or  $\operatorname{span}(e_2)\cap C^o \neq \emptyset$. Suppose $\operatorname{span}(e_1)\cap C^o \neq \emptyset$ is fulfilled, and $r\sin\,\theta>0$.

By Theorem \ref{theorem1} there exists a    quasi-eigenvectors $v_C(A) \in C$, and thus, $$
\langle A^Tv_C(A)- \underline{\lambda}_C(A)  v_C(A), w\rangle\leq  0,~\forall w \in C^o.
$$
Hence, $\langle A^Tv_C(A)- \underline{\lambda}_C(A)  v_C(A), e_1 \rangle\leq  0$, and consequently, $(r\cos\,\theta -\underline{\lambda}_C(A) )v_1+v_2 r \sin\,\theta  \leq 0$.
Since $v_1,v_2, r\sin\,\theta\geq 0$, we derive from here that $r\cos\,\theta\leq \underline{\lambda}_C(A) $. 
Similarly, using  $u_C(A) \in C$ we derive $\langle Au_C(A)- \overline{\lambda}_C(A)  u_C(A), e_1 \rangle\geq  0$, and consequently, $(r\cos\,\theta -\overline{\lambda}_C(A) )u_1-u_2 r \sin\,\theta \geq 0$. Hence, by $u_1,u_2, r\sin\,\theta\geq 0$, we obtain $r\cos\,\theta\geq \overline{\lambda}_C(A) $. Since $\overline{\lambda}_C(A) \geq \underline{\lambda}_C(A) $, we infer $\overline{\lambda}_C(A)=\underline{\lambda}_C(A)=r\cos\,\theta=\operatorname{Re} \lambda_{1}(\mathcal{O}_A)$.  The case $r\sin\,\theta<0$ is handled in the same way.



\begin{thebibliography}{1}


\bibitem{Arnold} V.I.~Arnol'd.
\newblock  {\em   Mathematical methods of classical mechanics (Vol. 60)}. 
\newblock Springer Science \& Business Media, 2013.

\bibitem{Bagirov} A. Bagirov,  N. Karmitsa. \& M. M. M\"akel\"a.
\newblock  {\em    Introduction
to Nonsmooth Optimization: Theory, Practice
and Software}.
\newblock Springer, 2014.

\bibitem{Barker}  G.P. Barker,  H. Schneider. \newblock Algebraic Perron-Frobenius theory. \newblock {\em Linear Algebra and its Applications}, 11(3):219--233, 1975.

\bibitem{Berman}  A. Berman,  R.J. Plemmons. \newblock {\em Nonnegative matrices in the mathematical sciences}. 
\newblock Society for Industrial and Applied Mathematics, 1994.


\bibitem{BirVar} G. Birkhoff, R.S. Varga. \newblock {\em Reactor Criticality and Non-negative Matrices (Vol. 166)}. 
\newblock Bettis Plant. 1957. 

\bibitem{Birkh} G. Birkhoff. \newblock Linear transformations with invariant cones. 
\newblock {\em Amer. Math. Monthly}, 74:274–276, 1967.


\bibitem{Burb} A.D. Burbanks, R.D. Nussbaum,  and C.T. Sparrow.   \newblock Extension of order-preserving maps on a cone. \newblock {\em Proceedings of the Royal Society of Edinburgh Section A: Mathematics}, 133(1):35-59, 2003.






\bibitem{bobkov} V. Bobkov, Y. Il’yasov. \newblock Maximal existence domains of positive solutions for two-parametric systems of elliptic equations. \newblock {\em Complex Variables and Elliptic Equations}, 61(5):587--607, 2016.

\bibitem{Chang} K.C. Chang, K. Pearson,  T. Zhang.  \newblock Perron-Frobenius theorem for nonnegative tensors. \newblock {\em Communications in Mathematical Sciences}, 6(2): 507-520 , 2008.

\bibitem{chen} R.L. Chen,  P.P. Varaiya.  \newblock Degenerate Hopf bifurcations in power systems. \newblock {\em IEEE transactions on circuits and systems}, 35(7): 818--824, 1988.


\bibitem{Friedland} S. Friedland. \newblock Characterizations of the spectral radius of positive operators. \newblock {\em Linear Algebra and its Applications}, 134:93--105, 1990.

\bibitem{FriedlandHan} S. Friedland, S. Gaubert, L. Han, \newblock  Perron–Frobenius theorem for nonnegative multilinear forms and extensions. \newblock {\em Linear Algebra and its Applications}, 438(2):738--749, 2013.

\bibitem{Gautier} A. Gautier, F. Tudisco,  M. Hein.   \newblock Nonlinear Perron--Frobenius Theorems for Nonnegative Tensors. \newblock {\em SIAM Review}, 65(2):495--536, 2023.

\bibitem{golub} G.H. Golub, and H.A.Van der Vorst. \newblock Eigenvalue computation in the 20th century. \newblock {\em Journal of Computational and Applied Mathematics}, 123(1-2):   35--65, 2000.



\bibitem{Hale} J. K. Hale, H. Koçak. \newblock {\em Dynamics and bifurcations. Vol. 3}. \newblock Springer Science \& Business Media, 2012.

\bibitem{Hassard} B.D. Hassard,  N.D. Kazarinoff and Y.H. Wan. \newblock {\em Theory and applications of Hopf bifurcation. Vol. 41}. \newblock CUP Archive, 1981.

\bibitem{IlyasFunc}
 Y.S. Il'yasov. \newblock Bifurcation calculus by the extended functional method,
 \newblock {\em Funct. Anal. Its Appl.},{41}(1):18--30, 2007.


\bibitem{IlyasPHD} Y. Il’yasov, \newblock A duality principle corresponding to the parabolic equations. \newblock {\em Physica D: Nonlinear Phenomena}, 237(2):692--698, 2008.



\bibitem{IvanIlya} 
 Y.S. Il'yasov,  A.A. Ivanov. \newblock Finding bifurcations for solutions of nonlinear equations by quadratic programming methods.  \newblock {\em Comp. Math. and Math. Phys.}, {53}(3):350--364, 2013.

\bibitem{IlIvan1}
 Y.S. Il'yasov,   A.A. Ivanov. \newblock Computation of maximal turning points to nonlinear   equations by nonsmooth optimization. \newblock {\em Optim. Meth. and   Softw.}, 31 (1):1--23,  2016.







\bibitem{IlyasChaos} 
Y. Il'yasov. \newblock Finding Saddle-Node Bifurcations via a Nonlinear Generalized Collatz–Wielandt Formula. \newblock {\em International Journal of Bifurcation and Chaos}. 31(01):2150008, 2021.


\bibitem{il2023_bifSYST} 
 Y. Il'yasov. \newblock A finding of the maximal saddle-node bifurcation for systems of differential equations. \newblock {\em J. Differ. Eqs}, 378: 610--625, 2024.


\bibitem{Krasnoselskii}  M.A. Krasnoselskii. \newblock {\em Positive solutions of operator equations}. \newblock 1st edition. 1964.


\bibitem{Krein} M.G. Krein, M. A. Rutman. \newblock Linear operators leaving invariant a cone in a Banach space. \newblock {\em Uspekhi mat. nauk}, 3:3--95, 1948.

\bibitem{Kuznet} Y.A. Kuznetsov. \newblock {\em Elements of Applied Bifurcation Theory, vol.112}. Springer Science \& Business Media, 2013.


\bibitem{lemmens} B. Lemmens,  R. Nussbaum. \newblock {\em Nonlinear Perron-Frobenius Theory (Vol. 189)}. \newblock Cambridge University Press, 2012.

\bibitem{marcus} M. Marcus, M. Henryk. \newblock {\em A survey of matrix theory and matrix inequalities. Vol. 14}. \newblock Courier Corporation, 1992.

\bibitem{Marek1} I. Marek. \newblock Spektrale Eigenschaften der K-positiven Operatoren and Einschliessungssatze fur den spektrahadius. \newblock {\em Czech. Math. J.}, 16:493--517, 1966.

\bibitem{Marek} I. Marek. \newblock Frobenius theory of positive operators: comparison theorems and applications. \newblock {\em SIAM 1. AppZ. Math}, 19:607--628, 1970.

\bibitem{Meyer} C.D. Meyer.  \newblock Continuity of the Perron root. \newblock {\em Linear and Multilinear Algebra}, 63(7):1332--1336, 2015.

\bibitem{Meyer2} C.D. Meyer, I. Stewart.  \newblock {\em Matrix analysis and applied linear algebra}. \newblock Society for Industrial and Applied Mathematics, 2023.

\bibitem{Mehrmann} V. Mehrmann,  H. Voss.  \newblock Nonlinear eigenvalue problems: A challenge for modern eigenvalue methods. \newblock {\em GAMM‐Mitteilungen}, 27(2):121--152, 2004.

\bibitem{parlet} B. N. Parlett, \newblock {\em The symmetric eigenvalue problem}. \newblock Society for Industrial and Applied Mathematics, 1998.

\bibitem{Salazar} P. D. P. Salazar, Y. Il'yasov, L. F. C. Alberto,  E. C. M. Costa \& M. B. Salles.  \newblock Saddle-node bifurcations of power systems in the context of variational theory and nonsmooth optimization, \newblock {\em IEEE Access} 8:110986--110993, 2020.

\bibitem{Pillai} S. U. Pillai, T. Suel, S. Cha. \newblock The Perron-Frobenius theorem: Some of its applications. \newblock {\em IEEE Signal Process. Magazine}, 22:62–-7, 2005.

\bibitem{Sion}  M. Sion.  \newblock {\em On general minimax theorems}.  171--176, 1958.

\bibitem{Tam}  B.S. Tam, S.F. Wu.  \newblock On the Collatz-Wielandt sets associated with a cone--preserving map. \newblock {\em Linear Algebra and its Applications}, 125:77--95, 1989.

\bibitem{Tam2} B.S. Tam, \newblock On the distinguished eigenvalues of a cone-preserving map. \newblock {\em Linear Algebra and its Applications}, 131:17--37,  1990.

\bibitem{Varga} R. S. Varga.  \newblock {\em Iterative analysis}. \newblock Prentice Hall, Englewood Cliffs, NJ, 1962.

\bibitem{Voss} H. Voss,  B. Werner.   \newblock A minimax principle for nonlinear eigenvalue problems with applications to nonoverdamped systems. \newblock {\em Mathematical Methods in the Applied Sciences}, 4(1):415--424 1982.

\bibitem{Weinberger} H.F. Weinberger.  \newblock {\em Variational methods for eigenvalue approximation}. \newblock Society for Industrial and Applied Mathematics, 1974.
\end{thebibliography}
\end{document}